\documentclass[12pt,a4paper]{amsart}
\usepackage[cp1250]{inputenc}

\usepackage{graphicx}
\usepackage{amsthm}
\usepackage{amsmath}
\usepackage{amssymb}

\theoremstyle{plain}
\newtheorem{theorem}{Theorem}[section]
\newtheorem{proposition}[theorem]{Proposition}
\newtheorem{lemma}[theorem]{Lemma}
\newtheorem{corollary}[theorem]{Corollary}

\theoremstyle{definition}
\newtheorem{definition}[theorem]{Definition}

\theoremstyle{remark}
\newtheorem{remark}[theorem]{Remark}

\renewcommand{\Re}{\operatorname{Re}}

\begin{document}

\pagestyle{plain}

\title{1-Grothendieck $C(K)$ spaces}
\author{Jind\v{r}ich Lechner}
\address{ Charles University in Prague\\
Faculty of Mathematics and Physics\\Department of Mathematical Analysis \\
Sokolovsk\'{a} 83, 186 \ 75\\Praha 8, Czech Republic}
\email{j.lechner@hush.com}

\thanks{This work was supported by the Grant No. 244214/B-MAT/MFF of the GAUK and by the Research grant GA\v{C}R P201/12/0290.}
\begin{abstract}
A Banach space is said to be Grothendieck if weak and weak$^*$ convergent sequences in the dual space coincide. This notion has been quantificated by H. Bendov\'{a}. She has proved that $\ell_\infty$ has the quantitative Grothendieck property, namely, it is 1-Grothendieck. Our aim is to show that Banach spaces from a certain wider class are 1-Grothendieck, precisely, $C(K)$ is 1-Grothendieck provided $K$ is a totally disconnected compact space such that its algebra of clopen subsets has the so called Subsequential completeness property.      
\end{abstract}
\maketitle

\section{Introduction and main results}

We say that a Banach space $X$ is \textit{Grothendieck} if each weak$^*$ convergent sequence in the dual space $X^{*}$ is necessarily weakly convergent. Naturally, every reflexive space is Grothendieck. Classical example of a nonreflexive Grothendieck space is $\ell_\infty$ due to Grothendieck~\cite{Grothendieck}. More generally, $C(K)$ is Grothendieck if $K$ is a compact Hausdorff $F$-space (i.e., disjoint open $F_{\sigma}$ subsets of $K$ have disjoint closures)~\cite{Seever}. According to R. Haydon~\cite[1B Proposition]{Haydon}, $C(K)$ is Grothendieck provided $K$ is a totally disconnected compact space such that its algebra of clopen subsets has the so called Subsequential completeness property. In~\cite{Haydon} Haydon has constructed such a space which moreover does not contain isomorphic copy of $\ell_\infty$. In~\cite{Pfitzner} H. Pfitzner has shown that each von Neumann algebra is a Grothendieck space. Some other Grothendieck spaces are the Hardy space $H^{\infty}$~\cite{Bourgain} or weak $L^p$ spaces~\cite{Lotz}.

The Grothendieck property has been quantificated by H. Bendov\'{a} in~\cite{Bendova} as follows 

\begin{definition} [the Quantitative Grothendieck property] Let $X$ be a Banach space. For a bounded sequence $(x_{n}^\ast)_{n\in\mathbb{N}}$ in the dual $X^\ast$ define two moduli:
               \begin{align*}
                   \delta_{w^\ast}(x_{n}^\ast)&:=\sup\left\{\mathrm{diam}\,\mathrm{clust}(x_{n}^\ast(x)):\;x\in B_{X}\right\}\text{,}\\
                   \delta_{w}(x_{n}^\ast)     &:=\sup\left\{\mathrm{diam}\,\mathrm{clust}(x^{\ast\ast}(x_{n}^\ast)):\;x^{\ast\ast}\in  B_{X^{\ast\ast}}\right\}\text{,}
               \end{align*}
where $\mathrm{clust}(a_n)$ with $(a_n)$ being a sequence denotes the set of all cluster points of $(a_n)$.
               Let $c\geq 1$. We say that $X$ is \textit{c-Grothendieck} if $\delta_{w}(x_{n}^\ast)\leq c\delta_{w^\ast}(x_{n}^\ast)$ whenever $(x_{n}^\ast)_{n\in\mathbb{N}}$ is a bounded sequence in $X^\ast$.
        \end{definition}	
It is known that $\ell_{\infty}$ is even 1-Grothendieck due to H. Bendov\'{a}~\cite[Theorem 1.1]{Bendova}. We generalize this result on a wider class of spaces. This class also includes the space which Haydon has constructed~\cite{Haydon}.

Now, let us remind the definitions of the above mentioned notions which were essential for Haydon's construction. 

\newpage
\begin{definition}{~}
\begin{enumerate}
\item We say that a topological space $T$ is totally disconnected if it contains at least two different points and each two different points are separated by a clopen set. 
\item We say that a totally disconnected compact space $K$ is a \textit{Haydon space} if the algebra
of its clopen subsets has \textit{the Subsequential completeness property} (SCP),
i.e., if for any sequence $(U_n)_{n\in\mathbb{N}}$ of pairwise disjoint clopen sets
there is an infinite set $M\subset\mathbb{N}$ such that the union of $(U_m)_{m\in\mathbb{M}}$ has the open closure.
\end{enumerate}
\end{definition}

Our aim is to show that  $C(K)$, that is $C(K;\mathbb{R})$ or $C(K;\mathbb{C})$, has the Quantitative Grothendieck property, namely it is 1-Grothendieck, provided $K$ is a Haydon space. Since the Quantitative Grothendieck property implies the Qualitative one, our result strengthens Haydon's proposition~\cite[1B Proposition]{Haydon}.

\begin{theorem} \label{Our_aim} If $S$ is a Haydon space then $C(S)$ is 1-Grothendieck. 
\end{theorem}

The proof of the theorem is in the section~\ref{real}. Since 1-Grothendieck property of $C(K;\mathbb{R})$ and $C(K;\mathbb{C})$ being equivalent we get our result for real and complex spaces at once. The equivalence is proved in the section~\ref{sec_equivalence}. 
				
\begin{corollary}
$C(K)$ is $1$-Grothendieck whenever $K$ is a $\sigma$-Stonean compact Hausdorff space (i.e., a compact Hausdorff space in which the closure of any open $F_\sigma$ set is open). In particular, $C(K)$ is $1$-Grothendieck whenever $K$ is an extremally disconnected (i.e., every open set has open closure) compact Hausdorff space.
\end{corollary}
\begin{proof}
In view of~\cite[Theorem A]{Seever} every $\sigma$-Stonean compact Hausdorff space is Haydon.
\end{proof}

\begin{corollary}
There is a nonreflexive $1$-Grothendieck space not containing $\ell_\infty$.
\end{corollary}
\begin{proof}
As we have already said Haydon had constructed a Haydon space $K$ with $C(K)$ not containing $\ell_\infty$~\cite{Haydon}. 
\end{proof}

\section{Real and complex case equivalence}\label{sec_equivalence}

This section is devoted to the following proposition.

\begin{proposition} \label{equivalence} Let $K$ be a compact Hausdorff space. Then the following assertions are equivalent:
\newcounter{ListCitac}
\begin{list}{(\roman{ListCitac})}{\usecounter{ListCitac}}
\item $C(K;\mathbb{R})$ is 1-Grothendieck.
\item $C(K;\mathbb{C})$ is 1-Grothendieck.
\item Whenever $\mu_n$ and $\nu_n$, $n\in\mathbb{N}$, are two sequences of Radon probability measures on $K$ such that $\mu_m$ and $\nu_n$ are mutually singular for each $m,\,n\in\mathbb{N}$ and $\varepsilon>0$ then there are $\Lambda\subset\mathbb{N}$ infinite and disjoint compact sets $A,\,B\subset T$ such that for each $n\in\Lambda$ we have $\mu_n(A)>1-\varepsilon$ and $\nu_n(B)>1-\varepsilon$.
\end{list}
\end{proposition}

To prove the equivalence of (i) and (ii) we need recall the notion of (I)-envelope established by O.F.K.~Kalenda~\cite{Kalenda} and the relationships between complex Banach spaces and real ones.	

\begin{definition}[(I)-envelope]
             Let $X$ be a Banach space and $B\subset X^\ast$. The (I)-envelope of $B$ is defined by the formula 
$$\mathrm{(I)}\mathrm{-env}(B):=\bigcap\left\{\overline{\mathrm{co}}^{\parallel\cdotp\parallel}\bigcup_{n=1}^{\infty}\overline{\mathrm{co}}^{w^{\ast}}C_{n}\,:\;B=\bigcup_{n=1}^{\infty}C_{n}\right\}\text{.}$$
\end{definition}

\begin{remark}

								\begin{description}
                     \item[$\diamond$] $\mathrm{(I)}\mathrm{-env}(B)$ is norm-closed and convex,
                     \item[$\diamond$] $\overline{\mathrm{co}}^{||\cdotp||}B\subset\mathrm{(I)}\mathrm{-env}(B)\subset\overline{\mathrm{co}}^{w^{\ast}}B$
                     \item[$\diamond$] $X$ is considered to be canonically embedded in $X^{\ast\ast}$, so the operation (I)-envelope applied to subsets of $X$ is done in the bidual $X^{\ast\ast}$.
                \end{description}	
\end{remark}

If $X$ is a complex Banach space then $X_{R}$ denotes $X$ considered as a real space. The following properties are well known and easy to check.
\begin{itemize}
\item The identity map $X$ onto $X_{R}$ is a real-linear isometry. Thus $B_{X}=B_{X_{R}}$.
\item The map $\phi:\,X^\ast\rightarrow \left(X_R\right)^{\ast}$ defined by $$\phi(x^\ast)(x)=\Re\,\left( x^\ast(x)\right),\quad x\in X,\,x^\ast\in X^\ast,$$ is a real-linear isometry.
\item The map $\psi:\,X^{\ast\ast}\rightarrow \left(X_R\right)^{\ast\ast}$ defined by $$ \psi(x^{\ast\ast})(x^\ast_R)=\Re\,\left( x^{\ast\ast}\left(\phi^{-1}(x^{\ast}_R)\right)\right),\quad x^{\ast}_R\in \left(X_R\right)^{\ast},\,x^{\ast\ast}\in X^{\ast\ast},$$ is a real-linear isometry and weak$^{\ast}$-to-weak$^{\ast}$ homeomorphism. This causes $\psi\left[\mathrm{(I)}\mathrm{-env}(B)\right]=\mathrm{(I)}\mathrm{-env}(\psi[B])$ whenever $B\subset X^{\ast\ast}$.
\item If $B\subset X$, then $\psi[\varepsilon_X[B]]=\varepsilon_{X_R}[B_R]$, where $B_R$ denotes $B$ considered as a subset of $X_R$ and $\varepsilon_X:X\rightarrow X^{\ast\ast}$ and $\varepsilon_{X_R}:X_{R}\rightarrow (X_{R})^{\ast\ast}$ are canonical embeddings into the respective biduals.
\end{itemize}

\begin{lemma} \label{pomoc} Let $X$ be a complex Banach space and $\phi:\,X^{\ast}\rightarrow \left(X_R\right)^{\ast}$ be as above. Then $\delta_{w^\ast}(x_{n}^\ast)=\delta_{w^\ast}\left(\phi (x_{n}^\ast)\right)$ and $\delta_w(x_{n}^\ast)=\delta_w\left(\phi (x_{n}^\ast)\right)$ for each bounded sequence $(x_{n}^\ast)_{n\in\mathbb{N}}$ in $X^\ast$. In particular, $X$ is c-Grothendieck if and only if $X_R$ is c-Grothendieck.
\end{lemma}
\begin{proof}
Let $(x_{n}^\ast)_{n\in\mathbb{N}}$ be a bounded sequence in $X^\ast$. We show only $\delta_w(x_{n}^\ast)=\delta_w\left(\phi (x_{n}^\ast)\right)$, the second equality $\delta_{w^\ast}(x_{n}^\ast)=\delta_{w^\ast}\left(\phi (x_{n}^\ast)\right)$ is done in the same way. Realize that $$\delta_w\left(\phi (x_{n}^\ast)\right)=\sup\left\{\mathrm{diam}\,\mathrm{clust}\left(\Re\,\left(x^{\ast\ast}\left(x_{n}^\ast\right)\right)\right):\,x^{\ast\ast}\in X^{\ast\ast}\right\}\text{.}$$ Whenever $a,b\in\mathbb{R}$ are cluster points of the sequence $\left(\Re\,\left(x^{\ast\ast}\left(x_{n}^{\ast}\right)\right)\right)_{n\in\mathbb{N}}$ for some $x^{\ast\ast}\in B_{X^{\ast\ast}}$, there are $c,d\in\mathbb{R}$ such that $a+ic$ and $b+id$ are cluster points of $(x^{\ast\ast}(x_{n}^\ast))_{n\in\mathbb{N}}$, thus $$\mathrm{diam}\,\mathrm{clust}\left(\Re\,\left(x^{\ast\ast}\left(x_{n}^\ast\right)\right)\right)\leq\mathrm{diam}\,\mathrm{clust}\left(x^{\ast\ast}\left(x_{n}^\ast\right)\right)\text{.}$$ And hence $\delta_w\left(\phi (x_{n}^\ast)\right)\leq\delta_w(x_{n}^\ast)$. To show the converse inequality, let $c$ such that $c<\delta_w(x_{n}^\ast)$. Then there exists $x^{\ast\ast}\in B_{X^{\ast\ast}}$ such that $\mathrm{diam}\,\mathrm{clust}\left(x^{\ast\ast}\left(x_{n}^\ast\right)\right)>c$, thus there are $a,b\in\mathrm{clust}\left(x^{\ast\ast}\left(x_{n}^\ast\right)\right)$ such that $|a-b|>c$. There are two increasing sequences of natural numbers $(p_k)_{k\in\mathbb{N}}$ and $(r_k)_{k\in\mathbb{N}}$ such that $$x^{\ast\ast}(x_{p_k}^\ast)\rightarrow a\,\text{ and }\,x^{\ast\ast}(x_{r_k}^\ast)\rightarrow b\text{ .} $$ Set $\alpha:=|a-b|/(a-b)$, then $\alpha$ is a complex unit, $\alpha(a-b)=|a-b|$ and $\alpha x^{\ast\ast}\in B_{X^{\ast\ast}}$. Then $$\Re\,\left(\left(\alpha x^{\ast\ast}\right)\left(x_{p_k}^{\ast}\right)\right)-\Re\,\left(\left(\alpha x^{\ast\ast}\right)\left(x_{r_k}^{\ast}\right)\right)=\Re\,\left(\alpha\left(x^{\ast\ast}\left(x_{p_k}^{\ast}\right)-x^{\ast\ast}\left(x_{r_k}^{\ast}\right)\right)\right)\rightarrow |a-b|\text{,}$$ since $\alpha\left(x^{\ast\ast}\left(x_{p_k}^{\ast}\right)-x^{\ast\ast}\left(x_{r_k}^{\ast}\right)\right)\rightarrow\alpha\left(a-b\right)=|a-b|$. Let $\breve{a}, \breve{b}\in\mathbb{R}$ be cluster points of $\left(\Re\,\left(\left(\alpha x^{\ast\ast}\right)\left(x_{p_k}^{\ast}\right)\right)\right)_{k\in\mathbb{N}}$, $\left(\Re\,\left(\left(\alpha x^{\ast\ast}\right)\left(x_{r_k}^{\ast}\right)\right)\right)_{k\in\mathbb{N}}$, respectively. Then $\breve{a}-\breve{b}=|a-b|>c$, thus $\mathrm{diam}\,\mathrm{clust}\left(\Re\,\left(\left(\alpha x^{\ast\ast}\right)\left(x_{n}^\ast\right)\right)\right)>c$ and $\delta_w\left(\phi (x_{n}^\ast)\right)>c$.
\end{proof}

\begin{proposition} \label{complex_Bendova} Let $X$ be a real or complex Banach space and $c\geq 1$. Then $X$ is c-Grothendieck if and only if $\mathrm{(I)}\mathrm{-env}(B_{X})\supset\frac{1}{c}B_{X^{\ast\ast}}$.
\end{proposition}
\begin{proof}
Real case is contained in~\cite[Proposition 2.2]{Bendova}. As far as complex case is concerned, we get the conclusion in consideration of $B_{X}=B_{X_{R}}$, $B_{\left(X_{R}\right)^{\ast\ast}}=\psi[B_{X^{\ast\ast}}]$, $\psi\left[\mathrm{(I)}\mathrm{-env}(B_{X})\right]=\mathrm{(I)}\mathrm{-env}(\psi[B_{X}])=\mathrm{(I)}\mathrm{-env}(B_{X_R})$ and Lemma~\ref{pomoc}.
\end{proof}

\begin{proof}[The proof of Proposition~\ref{equivalence}]
The equivalence of (i) and (iii) is obtained by the combination of \cite[Proposition 2.2]{Bendova} and \cite[Proposition 4.2]{Kalenda}.\\
$(i)\Leftrightarrow(ii):$	By~\cite[Proposition 2.2]{Bendova} we have the equivalence: $C(K;\mathbb{R})$ is 1-Grothendieck if and only if $\mathrm{(I)}\mathrm{-env}\left(B_{C(K;\mathbb{R})}\right)=B_{C(K;\mathbb{R})^{\ast\ast}}$. By~\cite[Proposition 5.1]{Kalenda} we get $\mathrm{(I)}\mathrm{-env}\left(B_{C(K;\mathbb{R})}\right)=B_{C(K;\mathbb{R})^{\ast\ast}}$ if and only if $\mathrm{(I)}\mathrm{-env}\left(B_{C(K;\mathbb{C})}\right)=B_{C(K;\mathbb{C})^{\ast\ast}}$ and according to Proposition~\ref{complex_Bendova} $\mathrm{(I)}\mathrm{-env}\left(B_{C(K;\mathbb{C})}\right)=B_{C(K;\mathbb{C})^{\ast\ast}}$ iff $C(K;\mathbb{C})$ is 1-Grothendieck.
\end{proof}

\section{The proper proof} \label{real}

\begin{definition} \label{def_mnozinove_funkce}
Let $\mathcal{S}$ be a family of sets such that $A\cup B\in\mathcal{S}$ whenever $A,\,B\in\mathcal{S}$. We say that a map $\varphi:\,\mathcal{S}\longrightarrow\mathbb{R}$ is
\begin{itemize}
\item \textit{superadditive} if $\varphi(A)+\varphi(B)\leq\varphi(A\cup B)$ whenever $A,\,B\in\mathcal{S}$ and $A\cap B=\emptyset$,
\item \textit{subadditive} if $\varphi(A\cup B)\leq\varphi(A)+\varphi(B)$ whenever $A,\,B\in\mathcal{S}$ and $A\cap B=\emptyset$,
\item \textit{additive} if $\varphi$ is superadditive and subadditive simultanously,
\item \textit{monotone} if $\varphi(A)\leq\varphi(B)$ whenever $A,\, B\in\mathcal{S}$ and $A\subset B$.
\end{itemize}
\end{definition}

\begin{remark} Let $\mathcal{S}$ be an algebra of sets and $\varphi:\,\mathcal{S}\longrightarrow\mathbb{R}$. 
\begin{enumerate}
\item If $\varphi$ is superadditive and attains only nonnegative values then is necessarily monotone.
\item If $\varphi$ is subadditive and monotone then $\varphi(A\cup B)\leq\varphi(A)+\varphi(B)$ whenever $A,\,B\in\mathcal{S}$.
\end{enumerate}
\end{remark}

From now on $\mathcal{U}(S)$ denotes the algebra of all clopen subsets of a totally disconnected compact space $S$.

\begin{lemma} \label{oddeleni_obojetnou_mnozinou_v_totalne_nesouvislem}
Let $S$ be a totally disconnected compact space. Whenever $A,\, B\subset S$ are disjoint closed sets then there exists a clopen set $M\subset S$ such that $A\subset M$ and $B\subset S\setminus  M$. 
\end{lemma}
\begin{proof}
Every compact Hausdorff space is normal, so two disjoint compact subsets can be separated by disjoint open sets. Moreover, these sets can be replaced by finite unions of basis elements and totally disconnected compact space has a basis consisting of clopen sets.
\end{proof}

The following two lemmas will be used in Lemma~\ref{Keylemma} (Key lemma). Both are in principle proved in~\cite[Lemma 4.4, 4.5]{Kalenda}. Although the assumptions in Lemma~\ref{Lemma1} are a little weaker (in ~\cite[Lemma 4.4]{Kalenda} there are additive functions), the proof is the same.

\begin{lemma}\label{Lemma1}
Let $\sigma_{n}$, $n\in\mathbb{N}$, be a sequence of nonnegative superadditive functions on $\mathcal{P}(\mathbb{N})$ satisfying for each $n\in\mathbb{N}$ the following special additive condition $$\sigma_{n}(A\cup B)=\sigma_{n}(A)+\sigma_{n}(B)$$ whenever $A\cap B=\emptyset$ and one of the sets $A$, $B$ is finite. Let further $N\subset\mathbb{N}$ infinite and $\varepsilon>0$ be such that $\sigma_{n}(F)<\varepsilon/2$ for each finite $F\subset N$ and for each $n\in\mathbb{N}$. Then there exists an infinite set $U\subset N$ such that $\sigma_{n}(U)<\varepsilon$ for all $n\in\mathbb{N}$.
\end{lemma}

\begin{lemma}\label{Lemma2}
Let $\lambda_{n}$, $n\in\mathbb{N}$, be a sequence of nonnegative additive functions on $\mathcal{P}(\mathbb{N})$ such that
\begin{enumerate}
\item $\lambda_{n}(\mathbb{N})\leq 1$ for each $n\in\mathbb{N}$,
\item $\lim_{n\rightarrow\infty}\lambda_{n}(\{k\in\mathbb{N}:\,k\geq n\})=0$.
\end{enumerate}
Then for each $\varepsilon>0$ there exists an increasing sequence of positive integers $p_0<p_1<p_2<\ldots$ such that for each infinite $U\subset\mathbb{N}$ we have $$\liminf_{n\rightarrow\infty}\left(\bigcup_{j\in\mathbb{N}\setminus U}\left\{k\in\mathbb{N}:\,p_j-1\leq k<p_j\right\}\right)\leq\varepsilon\text{.}$$
\end{lemma}

\begin{definition} [$\varepsilon$-separation, separation] Let $S$ be a totally disconnected compact space and $\mu$ and $\nu$ are nonnegative functions on $\mathcal{U}(S)$. If $\varepsilon>0$ and $A\in\mathcal{U(S)}$, we say that $A$ $\varepsilon$\textit{-separates} $\mu$ \textit{and} $\nu$ if $\mu(A)<\varepsilon$ and $\nu(S\setminus A)<\varepsilon$. Further, $\mu$ and $\nu$ are called $\mathcal{U(S)}$\textit{-separated} if for each $\varepsilon>0$ there is a clopen set $A\subset S$ which $\varepsilon$-separates $\mu$ and $\nu$. 
\end{definition}

\begin{definition} [Uniform separation] Let $S$ be a totally disconnected compact space and $\mu_{n}$ and $\nu_{n}$, $n\in\mathbb{N}$, are two sequences of nonnegative functions on $\mathcal{U}(S)$. We say for these two sequences to be \textit{uniformly} $\mathcal{U(S)}$\textit{-separated} if for each $\varepsilon>0$ there is a clopen set $A\subset S$ which $\varepsilon$-separates $\mu_m$ and $\nu_n$ for all $m\in\mathbb{N}$ and $n\in\mathbb{N}$. 
\end{definition}

From now on, $S$ is a Haydon space and $\tau(S)$ is the corresponding topology.

\begin{lemma}[Key lemma] \label{Keylemma}
Let $\mu_n$, $n\in\mathbb{N}$, be a sequence of nonnegative additive functions defined on $\mathcal{U}(S)$ and $\nu_n$, $n\in\mathbb{N}$, be a sequence of nonnegative monotone additive functions defined on $\tau(S)$. Further, assume that the following conditions hold
\vskip 0.25cm 
\begin{tabular}{rp{10cm}}
  $(i)$     & $\nu_n(S)=1$ for each $n\in\mathbb{N}$; \\
  $(ii)$    & for each $n\in\mathbb{N}$ and each $\varepsilon>0$ there exists $A\in\mathcal{U}(S)$ which $\varepsilon$-separates $\mu_m$ and $\nu_n$ for all $m\in\mathbb{N}$.
\end{tabular}
\vskip 0.25cm 
\noindent Then there exists an infinite set $\Lambda\subset\mathbb{N}$ such that for each $\varepsilon>0$ there is a set $Q\in\mathcal{U}(S)$ $\varepsilon$-separating $\mu_m$ and $\nu_n$ for all $m\in\mathbb{N}$ and all $n\in\Lambda$. 
\end{lemma}

\begin{proof}
We will proceed much as Kalenda did in the proof of~\cite[Lemma 4.6]{Kalenda}. First of all, we will show that for each $\varepsilon>0$ there exists an infinite set $\Lambda_0\subset\mathbb{N}$ and $A\in\mathcal{U}(S)$ such that $A$ $\varepsilon$-separates $\mu_m$ and $\nu_n$ for all $m\in\mathbb{N}$ and all $n\in\Lambda_0$. \paragraph{} 
Fix $\varepsilon>0$. According to the condition (ii), for each $n\in\mathbb{N}$ choose $A_n\in\mathcal{U}(S)$ such that $$\mu_m(A_n)<\varepsilon/2^{n+2}\quad \text{and}\quad \nu_n(S\setminus A_n)<\varepsilon/2^{n+2}$$ for all $m\in\mathbb{N}$. Set $C_{n}:=A_1\cup\ldots\cup A_n$, $n\in\mathbb{N}$. Then obviously $C_n\in\mathcal{U}(S)$, $n\in\mathbb{N}$. For each $m\in\mathbb{N}$, being additive on $\mathcal{U}(S)$, $\mu_m$ is monotone and subadditive and hence satisfies \begin{equation*} \mu_m(C_n)\leq\mu_m(A_1)+\ldots +\mu_m(A_n)<\varepsilon/2^3+\ldots+\varepsilon/2^{n+2}<\varepsilon/4\end{equation*}
for all $n\in\mathbb{N}$. Further, for each $n\in\mathbb{N}$, being monotone on $\mathcal{U}(S)$, $\nu_n$ satisfies $$\nu_n(S\setminus C_n)\leq\nu_n(S\setminus A_n)<\varepsilon/2^{n+2}<\varepsilon/4\text{.}$$
\paragraph{} For each $n\in\mathbb{N}$ set $$\lambda_n(D):=\nu_n\left(\bigcup_{k\in D}C_{k+1}\setminus C_k\right),\quad D\subset\mathbb{N}\text{.}$$
We check that the conditions in Lemma~\ref{Lemma2} hold for the sequence $\lambda_n$, $n\in\mathbb{N}$.
For each $n\in\mathbb{N}$ it holds $\lambda_{n}(\mathbb{N})\leq 1$, as $\nu_n$ is monotone and $\nu_n(S)=1$, and $\lambda_n$ is a nonnegative additive set function on $\mathbb{N}$, as $\nu_n$ is a nonnegative additive function on $\tau(S)$ and $\left(C_{k+1}\setminus C_{k}\right)\cap\left(C_{l+1}\setminus C_{l}\right)=\emptyset$ whenever $k\neq l$, $k,l\in\mathbb{N}$, and for $\nu_n$ is monotone for each $n\in\mathbb{N}$, we are able to get the condition (2) in Lemma~\ref{Lemma2} by the following estimates for each $n\in\mathbb{N}$
\begin{align*} \lambda_n\left(\left\{k\in\mathbb{N}:\,k\geq n\right\}\right)&=\nu_n\left(\bigcup_{k\geq n}C_{k+1}\setminus C_{k}\right)=\nu_n\left(\bigcup_{k\geq n}C_{k+1}\setminus C_n\right) \\
&\leq\nu_n\left(S\setminus C_{n}\right)<\varepsilon/2^{n+2}\text{.}
\end{align*}
According to Lemma~\ref{Lemma2}, there exists a sequence $p_0<p_1<\ldots$ of positive integers such that for each infinite set $U\subset\mathbb{N}$ we have \begin{equation*}\liminf_{n\rightarrow\infty}\lambda_n\left(\bigcup_{j\in\mathbb{N}\setminus U}\left\{k\in\mathbb{N}:\,p_{j-1}\leq k\leq p_j\right\}\right)\leq\varepsilon/2\text{.}\label{witness}\end{equation*}
\paragraph{} 
Further, for each $n\in\mathbb{N}$ set $$\sigma_n(D):=\mathrm{sup}\left\{\mu_n\left(\overline{\bigcup_{j\in E}C_{p_j}\setminus C_{p_{j-1}}}\right):\,E\subset D\;\&\;\overline{\bigcup_{j\in E}C_{p_j}\setminus C_{p_{j-1}}}\in\mathcal{U}(S)\right\}\text{,}$$ whenever $D\subset\mathbb{N}$. Note that if $D\subset\mathbb{N}$ is finite then $$\sigma_n(D)=\mu_n\left(\bigcup_{j\in D}C_{p_j}\setminus C_{p_{j-1}}\right),\;n\in\mathbb{N}\text{.}$$ Now we need to check for the sequence $\sigma_n$, $n\in\mathbb{N}$, to have the properties stated in the assumptions of Lemma~\ref{Lemma1}. Fix $n\in\mathbb{N}$. Then $\sigma_n$ is nonnegative as $\mu_n$ is nonnegative and bounded. To show that $\sigma_n$ is superadditive, choose $D_1$ and $D_2$ two disjoint subsets of $\mathbb{N}$. Fix $\delta>0$. Let $E_i\subset D_i$ be such that $\overline{\bigcup_{j\in E_i}C_{p_j}\setminus C_{p_{j-1}}}$ is open and $$\sigma_{n}(D_i)-\delta/2<\mu_n\left(\overline{\bigcup_{j\in E_i}C_{p_j}\setminus C_{p_{j-1}}}\right), \quad\quad i=1,2\text{.}$$ Set $E:=E_1\cup E_2$. Then $E\subset D_1\cup D_2$, $\overline{\bigcup_{j\in E}C_{p_j}\setminus C_{p_{j-1}}}$ is open and 
\begin{align*}
\sigma_n(D_1\cup D_2)&\geq\mu_n\left(\overline{\bigcup_{j\in E}C_{p_j}\setminus C_{p_{j-1}}}\right) \\
&=\mu_n\left(\overline{\bigcup_{j\in E_1}C_{p_j}\setminus C_{p_{j-1}}}\right)+\mu_n\left(\overline{\bigcup_{j\in E_2}C_{p_j}\setminus C_{p_{j-1}}}\right) \\
&\geq\sigma_n(D_1)+\sigma_n(D_2)-\delta\text{.}
\end{align*} Since $\delta>0$ is arbitrary, we get
$\sigma_n(D_1\cup D_2)\geq\sigma_n(D_1)+\sigma_n(D_2)$.
If $D_1$ is moreover finite, the equality holds. It is enough to show the converse inequality. Fix any $E\subset D_1\cup D_2$ such that $\overline{\bigcup_{j\in E}C_{p_j}\setminus C_{p_{j-1}}}$ is open. Set $E_1:=E\cap D_1$ and $E_2:=E\cap D_2$. Since $E_1$ is finite, $\bigcup_{j\in E_1}C_{p_j}\setminus C_{p_{j-1}}$ is closed (and hence clopen), thus $$\overline{\bigcup_{j\in E_2}C_{p_j}\setminus C_{p_{j-1}}}=\overline{\bigcup_{j\in E}C_{p_j}\setminus C_{p_{j-1}}}\setminus\bigcup_{j\in E_1}C_{p_j}\setminus C_{p_{j-1}}$$ is open as well. We have
\begin{align*}\mu_n\left(\overline{\bigcup_{j\in E}C_{p_j}\setminus C_{p_{j-1}}}\right)&=\mu_n\left(\bigcup_{j\in E_1}C_{p_j}\setminus C_{p_{j-1}}\right)+\mu_n\left(\overline{\bigcup_{j\in E_2}C_{p_j}\setminus C_{p_{j-1}}}\right) \\
&\leq\sigma_n(D_1)+\sigma_n(D_2)\text{.}
\end{align*} 
Since $E$ is arbitrary, we get $\sigma_n(D_1\cup D_2)\leq\sigma_n(D_1)+\sigma_n(D_2)$. 
Furthermore, for each $n\in\mathbb{N}$ and each finite $F\subset\mathbb{N}$ we have $$\sigma_n(F)=\mu_n\left(\bigcup_{j\in F}C_{p_j}\setminus C_{p_{j-1}}\right)<\varepsilon/4 \text{.}$$ According to Lemma~\ref{Lemma1}, there exists an infinite set $U\subset\mathbb{N}$ such that $\sigma_n(U)<\varepsilon/2$ for each $n\in\mathbb{N}$. Since $\mathcal{U}(S)$ has the SCP, there is infinite $V\subset U$ such that $\overline{\bigcup_{j\in V}C_{p_j}\setminus C_{p_{j-1}}}$ is open. And by definition of $\sigma_n$, $n\in\mathbb{N}$, we have for all $n\in\mathbb{N}$ $$\mu_n\left(\overline{\bigcup_{j\in V}C_{p_j}\setminus C_{p_{j-1}}}\right)\leq\sigma_n(U)<\varepsilon/2\text{.}$$ 
Set $A:=C_{p_0}\cup\overline{\bigcup_{j\in V}C_{p_j}\setminus C_{p_{j-1}}}$. Then $A\in\mathcal{U}(S)$ and for each $n\in\mathbb{N}$ we have 
$$\mu_n(A)=\mu_n(C_{p_0})+\mu_{n}\left(\overline{\bigcup_{j\in V}C_{p_j}\setminus C_{p_{j-1}}}\right)<\varepsilon/4+\varepsilon/2<\varepsilon\text{.}$$
Further, for all $n\in\mathbb{N}$ we have
\begin{align*}
\nu_n(S\setminus A)&\leq\nu_n\left(S\setminus C_n\right)+\nu_n\left(\bigcup_{j\in\mathbb{N}\setminus V}C_{p_j}\setminus C_{p_{j-1}}\right)\\
&<\varepsilon/4+\lambda_n\left(\bigcup_{j\in\mathbb{N}\setminus V}\left\{k\in\mathbb{N}:\,p_{j-1}\leq k<p_j\right\}\right)\text{.}
\end{align*}
It follows \begin{align*}\liminf_{n\rightarrow\infty}\nu_n\left(S\setminus A\right)\leq\varepsilon/4+\liminf_{n\rightarrow\infty}\lambda_n\left(\bigcup_{j\in\mathbb{N}\setminus V}\left\{k\in\mathbb{N}:\,p_{j-1}\leq k<p_j\right\}\right)<\varepsilon\text{.}\end{align*}
Hence, $\nu_n(S\setminus A)<\varepsilon$ for infinitely many $n\in\mathbb{N}$, that is to say, there is infinite set $\Lambda_0\subset\mathbb{N}$ such that $A$ $\varepsilon$-separates $\mu_m$ and $\nu_n$ for all $m\in\mathbb{N}$ and for all $n\in\Lambda_0$.\\
Now, by induction we construct an infinite sequence of infinite sets $\mathbb{N}\supset\Lambda_1\supset\Lambda_2\supset\ldots$ such that for each $k\in\mathbb{N}$ there exists $A_k\in\mathcal{U}(S)$ which $2^{-(k+1)}$-separates $\mu_m$ and $\nu_n$ for all $m\in\mathbb{N}$ and all $n\in\Lambda_k$. Pick pairwise distinct elements $c_{k}\in\Lambda_k$, $k\in\mathbb{N}$, and set $\Lambda:=\{c_k:\,k\in\mathbb{N}\}$. Then $\Lambda\subset\mathbb{N}$ is infinite. Choose arbitrary $\varepsilon>0$. Then there exists $k\in\mathbb{N}$ such that $2^{-k}<\varepsilon$. For each $j\in\{1,\ldots\ ,k-1\}$ choose $B_j\in\mathcal{U}(S)$ which $\frac{1}{k2^{k+1}}$-separates $\mu_{m}$ and $\nu_{c_j}$ for all $m\in\mathbb{N}$. Set $Q:=B_1\cup\ldots\cup B_{k-1}\cup A_k$. Obviously, $Q\in\mathcal{U}(S)$. For each $m\in\mathbb{N}$ we have $$\mu_{m}(Q)\leq\frac{k-1}{k2^{k+1}}+\frac{1}{2^{k+1}}<\frac{1}{2^{k}}<\varepsilon\text{.}$$
If $j\in\{1,\ldots\,k-1\}$, then $$\nu_{c_j}(S\setminus Q)\leq\nu_{c_j}(S\setminus B_j)<\frac{1}{k2^{k+1}}<\varepsilon\text{,}$$
and if $j\geq k$, then $$\nu_{c_j}(S\setminus Q)\leq\nu_{c_j}(S\setminus A_k)<2^{-(k+1)}<\varepsilon\text{.}$$ Hereby the proof is finished.
\end{proof}

\begin{lemma} [Inductive lemma] \label{Inductivelylemma}
Let $\mu_n$, $n\in\mathbb{N}$, be a sequence of nonnegative monotone additive functions defined on $\tau(S)$ and  $\nu_n$, $n\in\mathbb{N}$, be a sequence of nonnegative additive functions defined on $\mathcal{U(S)}$. Further, suppose that the following conditions are satisfied
\vskip 0.25cm 
\begin{tabular}{rp{10cm}}
  $(i)$      & $\mu_n(S)=1$ for each $n\in\mathbb{N}$; \\
  $(ii)$     & $\mu_m$ and $\nu_n$ are $\mathcal{U}(S)$-separated for each $m,\,n\in\mathbb{N}$.
\end{tabular}
\vskip 0.25cm 
\noindent Then there exists an infinite set $\Lambda\subset\mathbb{N}$ such that for each $n\in\mathbb{N}$ and each $\varepsilon>0$ there is a clopen set $P\subset S$ which $\varepsilon$-separates $\mu_m$ and $\nu_n$ for all $m\in\Lambda$.
\end{lemma}

\begin{proof} We will apply Lemma~\ref{Keylemma} inductively. Set $\tilde{\mu_{n}}:=\nu_{1}$ and $\tilde{\nu_{n}}:=\mu_{n}$ for each $n\in\mathbb{N}$. Sequences $\tilde{\mu_{n}}$ and $\tilde{\nu_{n}}$, $n\in\mathbb{N}$, satisfy assumptions of Lemma~\ref{Keylemma}. Thus, there is an infinite set $\Lambda_{1}\subset\mathbb{N}$ such that for each $\varepsilon>0$ there exists $Q\in\mathcal{U}(S)$ which $\varepsilon$-separates $\nu_1$ and $\mu_m$ for all $m\in\Lambda_{1}$. Having already constructed infinite set $\Lambda_{k}$ for $k\in\mathbb{N}$ such that for each $\varepsilon>0$ there exists $Q\in\mathcal{U}(S)$ which $\varepsilon$-separates $\nu_k$ and $\mu_m$ for all $m\in\Lambda_{k}$, set $\tilde{\mu_{n}}:=\nu_{k+1}$, $\tilde{\nu_{n}}:=\mu_{i_n}$ for each $n\in\mathbb{N}$ where $\Lambda_k=\{i_1<i_2<\ldots\}$. Then again according to Lemma~\ref{Keylemma} there exists infinite set $\tilde{\Lambda}\subset\mathbb{N}$ such that for each $\varepsilon>0$ there exists $Q\in\mathcal{U}(S)$ which $\varepsilon$-separates $\nu_{k+1}$ and $\mu_{i_m}$ for all $m\in\tilde{\Lambda}$. Set $\Lambda_{k+1}:=\{i_m:\,m\in\tilde{\Lambda}\}$. Then for each $k\in\mathbb{N}$ $\Lambda_{k+1}\subset\Lambda_{k}$ and for each $\varepsilon>0$ there exists $Q\in\mathcal{U}(S)$ which $\varepsilon$-separates $\nu_k$ and $\mu_m$ for all $m\in\Lambda_k$. Choose pairwise distinct elements $c_{n}\in\Lambda_{n}$, $n\in\mathbb{N}$, and set $\Lambda:=\{c_n:\,n\in\mathbb{N}\}$. Then $\Lambda$ is infinite. \\
Now fix $n\in\mathbb{N}$ and $\varepsilon>0$. Let $Q\in\mathcal{U}(S)$ $\varepsilon$/2-separate $\nu_n$ and $\mu_m$ for all $m\in\Lambda_n$ and since $\mathcal{U}(S)$-separation is a symmetric relation, from the condition $(ii)$ it follows that for each $i=1,\ldots,n-1$ there exists $C_i\in\mathcal{U}(S)$ $\varepsilon/2n$-separating $\nu_n$ and $\mu_{c_i}$. Set $A:=C_1\cup\ldots\cup C_{n-1}\cup Q$. Then from the additivity of $\nu_n$ on $\mathcal{U}(S)$ we get $$\nu_n(A)\leq\nu_n(C_1)+\ldots+\nu_n(C_1)+\nu_n(Q)\leq\frac{n-1}{n}\cdot\frac{\varepsilon}{2}+\frac{\varepsilon}{2}<\varepsilon\text{.}$$ 
If $j\in\{1,\ldots\,n-1\}$, then from the monotony of $\mu_{c_j}$ it implies $$\mu_{c_j}(S\setminus A)\leq\mu_{c_j}(S\setminus C_j)<\frac{\varepsilon}{2n}<\varepsilon\text{,}$$
and if $j\geq n$, then (again from the monotony of $\mu_{c_j}$) $$\mu_{c_j}(S\setminus A)\leq\mu_{c_j}(S\setminus Q)<\frac{\varepsilon}{2}<\varepsilon\text{.}$$ Setting $P:=S\setminus A$, we get the conclusion. 
\end{proof}

By synthesis of Lemma~\ref{Inductivelylemma} (Inductive lemma) and Lemma~\ref{Keylemma} (Key lemma) we get the following proposition.

\begin{proposition} \label{Substance}
Let $\mu_n$ and $\nu_n$, $n\in\mathbb{N}$, be sequences of nonnegative monotone additive functions defined on $\tau(S)$. Further, assume that the following conditions hold
\vskip 0.25cm 
\begin{tabular}{rp{10cm}}
  $(i)$      & $\mu_n(S)=1$ and $\nu_n(S)=1$ for each $n\in\mathbb{N}$; \\
  $(ii)$     & $\mu_m$ and $\nu_n$ are $\mathcal{U}(S)$-separated for each $m,\,n\in\mathbb{N}$.
\end{tabular}
\vskip 0.25cm 
\noindent Then there exists an infinite set $\Lambda\subset\mathbb{N}$ such that for each $\varepsilon>0$ there is a clopen set $A\subset S$ which $\varepsilon$-separates $\mu_m$ and $\nu_n$ for all $m\in\Lambda$ and $n\in\Lambda$. In other words,  there exists an infinite set $\Lambda\subset\mathbb{N}$ such that the sequences $\mu_n$ and $\nu_n$, $n\in\Lambda$, are uniformly $\mathcal{U}(S)$-separated.
\end{proposition}
\begin{proof}
According to Lemma~\ref{Inductivelylemma} there is an infinite set $\Lambda_1\subset\mathbb{N}$ such that for each $n\in\mathbb{N}$ and each $\varepsilon>0$ there is a clopen set $P\subset S$ which $\varepsilon$-separates $\mu_m$ and $\nu_n$ for all $m\in\Lambda_1$. Thus, the condition $(ii)$ in Lemma~\ref{Keylemma} is satisfied for sequences $\mu_{i_n}$ and $\nu_{i_n}$, $n\in\mathbb{N}$, where $\Lambda_1=\{i_1<i_2<\ldots\}$. According to Lemma~\ref{Keylemma} there exists an infinite set $\Lambda_2\subset\mathbb{N}$ such that for each $\varepsilon>0$ there is a set $Q\in\mathcal{U}(S)$ $\varepsilon$-separating $\mu_{i_m}$ and $\nu_{i_n}$ for all $m\in\mathbb{N}$ and $n\in\Lambda_2$. Set $\Lambda:=\{i_n:\,n\in\Lambda_2\}$.
\end{proof}

\begin{proof} [Proof of Theorem~\ref{Our_aim}] To prove that $C(S)$ is 1-Grothendieck it suffices the condition (iii) in Proposition~\ref{equivalence}. So let $\mu_n$ and $\nu_n$, $n\in\mathbb{N}$, be two sequences of Radon probability measures on $S$ such that $\mu_m$ and $\nu_n$ are mutually singular for each $m,\,n\in\mathbb{N}$.\par Now fix $m,\,n\in\mathbb{N}$ and $\varepsilon>0$. From mutually singularity and regularity of $\mu_m$ and $\nu_n$ and also from the fact that they are probabilities there exist two disjoint closed sets $K,\,L\subset S$ such that $\mu_m(K)>1-\varepsilon$ and $\nu_n(L)>1-\varepsilon$. According to Lemma~\ref{oddeleni_obojetnou_mnozinou_v_totalne_nesouvislem} there is a clopen set $M\subset S$ such that $L\subset M$ and $K\subset S\setminus M$. Thus, $\mu_m(M)<\varepsilon$ and $\nu_n(M)>1-\varepsilon$, that is $\nu_n(S\setminus M)<\varepsilon$. We have just shown that $\mu_m$ and $\nu_n$ are $\mathcal{U}(S)$-separated for each $m,\,n\in\mathbb{N}$. \par
The conditions of Proposition~\ref{Substance} are satisfied. There is the condition (iii) we have checked above and the others are clear. Proposition~\ref{Substance} says that there is an infinite set $\Lambda\subset\mathbb{N}$ such that for each $\varepsilon>0$ there exists a clopen set $A\subset S$ which $\varepsilon$-separates $\mu_m$ and $\nu_n$ for all $m,\,n\in\Lambda$.
\end{proof}

We still do not know if the other Grothendieck spaces mentioned in the section \textit{Introduction and main results} are 1-Grothendieck as well. 


\end{document}